\documentclass{amsart}
\usepackage{amsmath}
\usepackage{amssymb}
\usepackage{amsfonts}
\usepackage{color}
\usepackage[dvips]{graphicx}
\usepackage[all,knot,poly,tpic]{xy}

\newtheorem{thm}{Theorem}
\newtheorem{lem}[thm]{Lemma}

\newtheorem{cor}[thm]{Corollary}
\newtheorem*{expl}{Example}
\newtheorem*{thm41}{Theorem 4.1, \cite{CFP}}
\newtheorem*{thm43}{Theorem 4.3, \cite{CFP}}
\newtheorem*{thm46}{Theorem 4.6, \cite{CFP}}
\newtheorem*{thm48}{Theorem 4.8, \cite{CFP}}
\newtheorem*{cor47}{Corollary 4.7, \cite{CFP}}
\newtheorem*{cor49}{Corollary 4.9, \cite{CFP}}

\def\xa{\xy 0*=\txt{$x_0$},
 {+(6,-4)\PATH~={**@{-}}'+(0,8)'+(64,0)'-(0,8)'-(64,0)},
 {+(16,0)*{} \ar @{-} +(16,8)*{}}, {+(16,0)*{} \ar @{-}
 +(16,8)*{}},-(32,0)
 \endxy}

\def\xb{\xy 0*=\txt{$x_1$},
 {+(6,-4)\PATH~={**@{-}}'+(0,8)'+(64,0)'-(0,8)'-(64,0)},
 {+(32,0)*{}\ar @{-} +(0,8)*{}}, {+(8,0)*{}\ar @{-} +(8,8)*{}},
 {+(8,0)*{}\ar @{-} +(8,8)*{}},-(32,0)
 \endxy}

\def\xc{\xy
 0*=\txt{$x_0^{-1}$},
  {+(6,-4)\PATH~={**@{-}}'+(0,8)'+(64,0)'-(0,8)'-(64,0)},
  {+(32,0)*{} \ar @{-} +(-16,8)*{}},
  {+(16,0)*{} \ar @{-} +(-16,8)*{}},
  {+(4,0)*{} \ar @{.} +(-12,8)*{}},
  {+(4,0)*{} \ar @{.} +(-8,8)*{}},
 +(-62,12)*=\txt{$x_1$},
  {+(6,-4)\PATH~={**@{-}}'+(0,8)'+(64,0)'-(0,8)},
  {+(16,0)*{}\ar @{.} +(0,8)*{}},
  {+(16,0)*{}\ar @{-} +(0,8)*{}},
  {+(8,0)*{}\ar @{-} +(8,8)*{}},
  {+(8,0)*{}\ar @{-} +(8,8)*{}}
 \endxy}

\def\xd{\xy 0*=\txt{$x_0^{-1}x_1$},
  {+(6,-4)\PATH~={**@{-}}'+(0,8)'+(64,0)'-(0,8)'-(64,0)},
  {+(32,0)*{}\ar @{-} +(-16,8)*{}},
  {+(16,0)*{}\ar @{-} +(-16,8)*{}},
  {+(4,0)*{}\ar @{-} +(-4,8)*{}},
  {+(4,0)*{}\ar @{-} +(0,8)*{}}
 \endxy}

\def\kol{2}

\def\xe{\xy 0*{},
  {+(6,4)*{}\ar @{-} +(64,0)*{}},
  {+(0,0)*{}\ar @{-} +(0,\kol)*{}},
  {+(.5,0)*{}\ar @{-} +(0,\kol)*{}},
  {+(.5,0)*{}\ar @{-} +(0,\kol)*{}},
  {+(1,0)*{}\ar @{-} +(0,\kol)*{}},
  {+(2,0)*{}\ar @{-} +(0,\kol)*{}},
  {+(4,0)*{}\ar @{-} +(0,\kol)*{}},
  {+(8,0)*{}\ar @{-} +(0,\kol)*{}},
  {+(16,0)*{}\ar @{-} +(0,\kol)*{}},
  {+(16,0)*{}\ar @{-} +(0,\kol)*{}},
  {+(8,0)*{}\ar @{-} +(0,\kol)*{}},
  {+(4,0)*{}\ar @{-} +(0,\kol)*{}},
  {+(2,0)*{}\ar @{-} +(0,\kol)*{}},
  {+(1,0)*{}\ar @{-} +(0,\kol)*{}},
  {+(.5,0)*{}\ar @{-} +(0,\kol)*{}},
  {+(.5,0)*{}\ar @{-} +(0,\kol)*{}},
  (6,0)*=\txt{$0$},
  +(32,0)*=\txt{$\frac{1}{2}$},
  +(32,0)*=\txt{$1$},
 \endxy}

\def\xf{\xy 0*=\txt{$f$},
  {+(6,-4)*{}\ar @{-} +(64,0)*{}},
  {+(0,8)*{}\ar @{-} +(64,0)*{}},
  (6,-4),
  {+(10,0)*{}\ar @{-} +(4,8)*{}},
  +(14,0),
  +(10,0),
  +(10,0),
  {+(10,0)*{}\ar @{-} +(-16,8)*{}},
  (32,0)*=\txt{$?$},
  (20,7)*=\txt{$x_i$},
  +(14,0.5),
  +(14,-0.5)*=\txt{$x_{i+1}$},
  (16,-7)*=\txt{$y_i$},
  +(14,0.5),
  +(30,-0.5)*=\txt{$y_{i+1}$},
 \endxy}

\def\xg{\xy 0*=\txt{$f$},
  {+(6,-4)*{}\ar @{-} +(64,0)*{}},
  {+(0,8)*{}\ar @{-} +(64,0)*{}},
  (6,-4),
  {+(10,0)*{}\ar @{-} +(4,8)*{}},
  {+(14,0)*{}\ar @{-} +(4,8)*{}},
  {+(10,0)*{}\ar @{-} +(0,1)*{}},
  {+(10,0)*{}\ar @{-} +(0,1)*{}},
  {+(10,0)*{}\ar @{-} +(-16,8)*{}},
  (25,0)*=\txt{$1$},
  (40,0)*=\txt{$?$},
  (20,7)*=\txt{$x_i$},
  +(14,0.5)*=\txt{$x_i+d$},
  +(14,-0.5)*=\txt{$x_{i+1}$},
  (16,-7)*=\txt{$y_i$},
  +(14,0.5)*=\txt{$y_i+d$},
  +(30,-0.5)*=\txt{$y_{i+1}$},
 \endxy}

\def\xh{\xy 0*=\txt{$f$},
  {+(6,-4)*{}\ar @{-} +(64,0)*{}},
  {+(0,8)*{}\ar @{-} +(64,0)*{}},
  (6,-4),
  {+(10,0)*{}\ar @{-} +(4,8)*{}},
  {+(14,0)*{}\ar @{-} +(4,8)*{}},
  {+(5,0)*{}\ar @{-} +(0,1)*{}},
  {+(5,0)*{}\ar @{-} +(0,1)*{}},
  {+(5,0)*{}\ar @{-} +(0,1)*{}},
  {+(5,0)*{}\ar @{-} +(0,1)*{}},
  {+(0,0)*{}\ar @{-} +(-11,8)*{}},
  {+(5,0)*{}\ar @{-} +(0,1)*{}},
  {+(5,0)*{}\ar @{-} +(-16,8)*{}},
  (25,0)*=\txt{$1$},
  (38,0)*=\txt{$4$},
  (48,0)*=\txt{$2$},
  (20,7)*=\txt{$x_i$},
  +(14,0.5)*=\txt{$x_i+d$},
  +(14,-0.5)*=\txt{$x_{i+1}$},
  (16,-7)*=\txt{$y_i$},
  +(14,0.5)*=\txt{$y_i+d$},
  +(30,-0.5)*=\txt{$y_{i+1}$},
 \endxy}

\def\xi{\xy 0*=\txt{},
  {+(0,-4)*{}\ar @{-} +(28,0)*{}},
  {+(0,8)*{}\ar @{-} +(28,0)*{}},
  +(0,-8),
  {+(0,0)*{}\ar @{~} +(0,8)*{}},
  -(0,3.5)*=\txt{$0$},+(0,3.5),
  {+(8,0)*{}\ar @{.} +(0,8)*{}},
  -(0,3.5)*=\txt{$\beta_1$},+(0,3.5),
  {+(8,0)*{}\ar @{.} +(0,8)*{}},
  -(0,3.5)*=\txt{$\beta_2$},+(0,3.5),
  {+(8,0)*{}\ar @{.} +(0,8)*{}},
  -(0,3.5)*=\txt{$\beta_3$},+(0,3.5),
  +(8,4)*=\txt{$\dots$},
  {+(4,-4)*{}\ar @{-} +(20,0)*{}},
  {+(0,8)*{}\ar @{-} +(20,0)*{}},
  +(0,-8),
  {+(4,0)*{}\ar @{.} +(0,8)*{}},
  -(0,3.5)*=\txt{$\beta_k$},+(0,3.5),
  {+(8,0)*{}\ar @{.} +(0,8)*{}},
  -(0,3.5)*=\txt{$\beta_{k+1}$},+(0,3.5),
  {+(8,0)*{}\ar @{-} +(0,8)*{}},
  -(0,3.5)*=\txt{$1$},+(0,3.5)
 \endxy}

\def\xj{\xy 0*=\txt{$a$},
  {+(10,-4)*{}\ar @{-} +(48,0)*{}},
  {+(0,8)*{}\ar @{-} +(48,0)*{}},
  +(0,-8),
  {+(0,0)*{}\ar @{~} +(0,8)*{}},
  -(0,3.5)*=\txt{$0$},+(0,3.5),
  {+(8,0)*{}\ar @{.} +(0,8)*{}},
  -(0,3.5)*=\txt{$\beta_1$},+(0,3.5),
  {+(8,0)*{}\ar @{.} +(0,8)*{}},
  {+(0,0)*{}\ar @{-} +(-8,8)*{}},
  -(0,3.5)*=\txt{$\beta_2$},+(0,3.5),
  {+(8,0)*{}\ar @{.} +(0,8)*{}},
  {+(0,0)*{}\ar @{-} +(8,8)*{}},
  -(0,3.5)*=\txt{$\beta_3$},+(0,3.5),
  {+(8,0)*{}\ar @{.} +(0,8)*{}},
  -(0,3.5)*=\txt{$\beta_4$},+(0,3.5),
  {+(8,0)*{}\ar @{-} +(0,8)*{}},
  -(0,3.5)*=\txt{$\beta_5$},+(0,3.5),
  {+(8,0)*{}\ar @{-} +(0,8)*{}},
  -(0,3.5)*=\txt{$1$},+(0,3.5),
 \endxy}

\def\xk{\xy 0*=\txt{$b$},
  {+(10,-4)*{}\ar @{-} +(48,0)*{}},
  {+(0,8)*{}\ar @{-} +(48,0)*{}},
  +(0,-8),
  {+(0,0)*{}\ar @{~} +(0,8)*{}},
  -(0,3.5)*=\txt{$0$},+(0,3.5),
  {+(8,0)*{}\ar @{-} +(0,8)*{}},
  -(0,3.5)*=\txt{$\beta_1$},+(0,3.5),
  {+(8,0)*{}\ar @{.} +(0,8)*{}},
  -(0,3.5)*=\txt{$\beta_2$},+(0,3.5),
  {+(8,0)*{}\ar @{.} +(0,8)*{}},
  {+(0,0)*{}\ar @{-} +(-8,8)*{}},
  -(0,3.5)*=\txt{$\beta_3$},+(0,3.5),
  {+(8,0)*{}\ar @{.} +(0,8)*{}},
  {+(0,0)*{}\ar @{-} +(8,8)*{}},
  -(0,3.5)*=\txt{$\beta_4$},+(0,3.5),
  {+(8,0)*{}\ar @{.} +(0,8)*{}},
  -(0,3.5)*=\txt{$\beta_5$},+(0,3.5),
  {+(8,0)*{}\ar @{-} +(0,8)*{}},
  -(0,3.5)*=\txt{$1$},+(0,3.5)
 \endxy}

\def\xl{\xy 0*=\txt{$b^{-1}$},
  +(0,8)*=\txt{$a^{-1}$},
  +(0,8)*=\txt{$b$},
  +(0,8)*=\txt{$a$},
  {+(10,-28)*{}\ar @{-} +(48,0)*{}},
  {+(0,8)*{}\ar @{-} +(48,0)*{}},
  {+(0,8)*{}\ar @{-} +(48,0)*{}},
  {+(0,8)*{}\ar @{-} +(48,0)*{}},
  {+(0,8)*{}\ar @{-} +(48,0)*{}},
  {+(0,-32.5)*{}\ar @{~} +(0,32.5)*{}}, +(0,0.5),
  -(0,3.5)*=\txt{$0$},+(0,3.5),
  {+(8,0)*{}\ar @{.} +(0,32)*{}},
  {+(0,0)*{}\ar @{-} +(0,8)*{}},
  {+(0,8)*{}\ar @{-} +(8,8)*{}},
  {+(0,8)*{}\ar @{-} +(0,8)*{}},-(0,16)
  -(0,3.5)*=\txt{$\beta_1$},+(0,3.5),
  {+(8,0)*{}\ar @{.} +(0,32)*{}},
  {+(0,0)*{}\ar @{-} +(8,8)*{}},
  -(0,3.5)*=\txt{$\beta_2$},+(0,3.5),
  {+(8,0)*{}\ar @{.} +(0,32)*{}},
  {+(0,24)*{}\ar @{-} +(8,8)*{}},-(0,24),
  -(0,3.5)*=\txt{$\beta_3$},+(0,3.5),
  {+(8,0)*{}\ar @{.} +(0,32)*{}},
  {+(0,16)*{}\ar @{-} +(8,8)*{}},-(0,16),
  -(0,3.5)*=\txt{$\beta_4$},+(0,3.5),
  {+(8,0)*{}\ar @{.} +(0,32)*{}},
  {+(0,8)*{}\ar @{-} +(0,8)*{}},-(0,8),
  {+(0,24)*{}\ar @{-} +(0,8)*{}},-(0,24),
  {+(0,0)*{}\ar @{-} +(-32,32)*{}},
  -(0,3.5)*=\txt{$\beta_5$},+(0,3.5),
  {+(8,0)*{}\ar @{-} +(0,32)*{}},
  -(0,3.5)*=\txt{$1$},+(0,3.5)
 \endxy}

\title{On identities in Thompson's group}

\author[E. Esyp]{Evgenii S. Esyp}
\address{Evgenii Esyp,
Omsk Branch of Institute of Mathematics (SB RAS),
13 Pevtsova St., Omsk, 644099, Russia}
\email{esyp@iitam.omsk.net.ru}

\begin{document}
\maketitle

\section{Introduction}
Thompson's group $F$ is the group of all  piecewise linear orientation-preserving homeomorphisms of the unit interval $[0, 1]$ with finitely many breakpoints such that:
\begin{enumerate}
 \item all slopes are powers of $2$;
 \item all breakpoints are in $\mathbb{Z}[\frac{1}{2}]$, the ring of dyadic rational numbers.
\end{enumerate}

The group $F$ is described by the following presentation:
$$
F = \langle x_0, x_1, x_2 \dots \mid x_i ^ {-1} x_k x_i=x_{k+1} (k> i) \rangle.
$$
One can show that the group $F$ can also be given by the following presentation
$$
F=\langle x_0, x_1 \mid [x_0x_1 ^ {-1}, x_0 ^ {-1} x_1x_0] = [x_0x_1 ^ {-1}, x_0 ^ {-2} x_1x_0^2] =1\rangle.
$$

The Thompson's group $F$ posses many interesting properties, among which we shall only mention a few. We refer the reader to \cite{CFP} and references there for a detailed account.

\begin{lem} \label{lem1}
If $0=x_0<x_1<x_2<\dots <x_n=1$ and $0=y_0<y_1<y_2<\dots y_n=1$ are two partitions of $[0,1]$ consisting of dyadic rational numbers, then there exists a piecewise linear homeomorphism $f$ of the unit interval $[0,1]$ such that $f(x_i)=y_i$ for $i=0,\dots ,n$, and $f$ is an element of $F$.
\end{lem}

We shall make use of the following results on the Thompson's group by $F$, which can be found in \cite{CFP}.

\begin{thm41}
The derived subgroup $[F,F]$ of $F$ consists of all elements in $F$ which are trivial in neighborhoods of $0$ and $1$. Furthermore, $F/[F,F]\cong\mathbb{Z}\oplus\mathbb{Z}$
\end{thm41}

\begin{thm43}
Every proper quotient group of $F$ is Abelian.
\end{thm43}

\begin{thm46}
The submonoid of $F$ generated by $x_0$, $x_1$, $x_1^{-1}$ is the free product of the submonoid generated by $x_0$ and the subgroup generated by $x_1$.
\end{thm46}

\begin{cor47}
Thompson's Group $F $ has exponential growth.
\end{cor47}

\begin{thm48}
Every non-Abelian subgroup of $F$ contains a free Abelian subgroup of infinite rank.
\end{thm48}

\begin{cor49}
Thompson's group $F$ does not contain a non-Abelian free group.
\end{cor49}

The aim of this paper is to prove that Thompson's group $F$ is close to a free group in that it does not satisfy an identity.

\begin{thm}\label{thm1}
For any natural numbers $n$ and $k$ there exist $n$ elements $a_1, \dots, a_n$ in Thompson's group $F$, such that no relation involving $n$ variables and of length less than $k$ is satisfied by $a_1, \dots, a_n$.
\end{thm}

In particular, if $n=2 $ we get the following

\begin{cor}
For any positive integer $k$ there exists a pair of elements $a,b$ in Thompson's group $F$, such that no relation of length less than $k$ is satisfied by $a,b$.
\end{cor}

Our proof of Theorem \ref{thm1} consists of the following three lemmas

\begin{lem}\label{lem2}
Let $G$ be a finitely generated group. For any positive integer $n$ the following two conditions are equivalent:
\begin{enumerate}
 \item \label{lem2it1}for any positive integer $k$ there exist $n$ elements  $a_{k,1},\dots,a_{k,n}$ of $G$ such that the tuple $a_{k,1},\dots,a_{k,n}$ does not satisfy any relation of length less than $k$;
 \item \label{lem2it2}the group $G$ satisfies no identity in $n$ variables.
\end{enumerate}
\end{lem}

\begin{lem}\label{lem3}
Let $G$ be a group. If the group $G$ satisfies an identity $f(x_1, \dots, x_n)$ in $n$ variables, then the group $G$ satisfies an identity $g$ in $2$ variables.
\end{lem}
\begin{proof}
Let $f$ be an identity on $G$. Set $g (x, y) =f(w_1 (x, y), w_2 (x, y), \dots, w_n (x, y))$, where the words
$w_1 (x, y), w_2 (x, y), \dots, w_n (x, y)$ generate the free subgroup of rank $n$ in the free group $F(x,y)$. It follows that $g(x,y)$ is a nontrivial identity.
\end{proof}

\begin{lem}\label{lem4}
Thompson's group $F$ does not satisfy an identity.
\end{lem}

\section{Some examples and definitions}
Let $x_0$ and $x_1$ be the following elements of the group $F$:
$$
x_0 (x) = \left \{
 \begin {array} {ll}
 x/2, & 0 \leq x \leq 1/2 \\
 x-1/4, & 1/2 \leq x \leq 3/4 \\
 2x-1, & 3/4 \leq x \leq 1 \\
 \end {array}
 \right.
$$
$$
x_1 (x) = \left \{
 \begin {array} {ll}
 x, & 0 \leq x \leq 1/2 \\
 x/2+1/4, & 1/2 \leq x \leq 3/4 \\
 x-1/8, & 3/4 \leq x \leq 7/8 \\
 2x-1, & 7/8 \leq x \leq 1 \\
 \end {array}
 \right.
$$
It is convenient to represent the elements of Thompson's group $F$ as rectangular diagrams. For example, $x_0$ and $x_1$ can be represented as follows:
$$
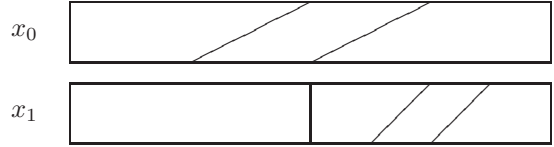
\xa
$$
$$
\xb
$$
Here the upper side of the rectangle is mapped onto the lower one. For instance, take the product of $x_1$ and $x_0^{-1}$:
$$
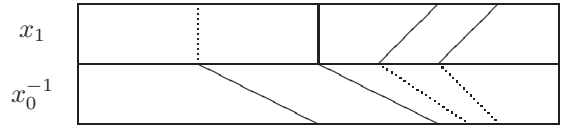
\xc
$$
Then the rectangular diagram for $x_1x_0^{-1}$ has the form:
$$
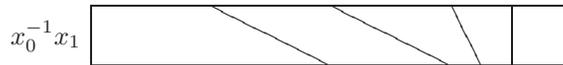
\xd
$$

Let us construct a special partition of the unit interval $[0,1]$. We subdivide $[0,1]$ into infinitely many pieces of the form $[1/2 ^ {k+1}, 1/2^k]$,
$[1-1/2^k, 1-1/2 ^ {k+1}]$, where $k \geq 1$ is an integer. This subdivision is shown on the figure below:
$$
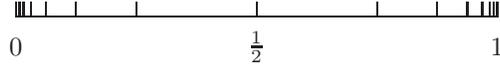
\xe
$$

The element $x_0$ shifts these pieces of this subdivision to the right:
$$
\dots,  [\frac{1}{8},\frac{1}{4}]\rightarrow [\frac{1}{4},\frac{1}{2}], \ [\frac{1}{4},\frac{1}{2}]\rightarrow [\frac{1}{2},\frac{3}{4}], \
[\frac{1}{2},\frac{3}{4}]\rightarrow [\frac{3}{4},\frac{7}{8}],\ [\frac{3}{4},\frac{7}{8}]\rightarrow [\frac{7}{8},\frac{15}{16}], \dots
$$

The element $x_0^{-1}$ shifts the pieces of the subdivision to the left:
$$
\dots, [\frac{1}{8},\frac{1}{4}]\leftarrow [\frac{1}{4},\frac{1}{2}],\ [\frac{1}{4},\frac{1}{2}]\leftarrow [\frac{1}{2},\frac{3}{4}],\ [\frac{1}{2},\frac{3}{4}]\leftarrow [\frac{3}{4},\frac{7}{8}],\ [\frac{3}{4},\frac{7}{8}]\leftarrow [\frac{7}{8},\frac{15}{16}],\ \dots
$$

\section{Proof of Lemma \ref{lem2}}

We first prove that (\ref{lem2it1}) implies (\ref{lem2it2}). Take an arbitrary reduced word $f$ in $n$ variables and of length $l$. Then there exist elements $a_{l+1,1}, \dots, a_{l+1,n} \in G$ such that $f (a_{l+1,1}, \dots, a_{l+1, n}) \neq 1 $. Hence $f$ is not an identity.

We now prove the converse. Let us assume, that the group $G$ does not satisfy any identities. Assume the contrary, i.e. that condition (\ref{lem2it1}) from Lemma \ref{lem2} does not hold. Then there exist natural numbers $n$ and $k$, such that any $n$ elements $a_1,\dots,a_n\in G$ satisfy a relation of length less then $k$. Consider the set of all such relations, that is the set of all reduced words (one can treat these words as elements of a free group) $B_k = \{f_1, \dots, f_{m_k}\}$ in $n$ letters and of length less than $k$.

Using the set $B_k$ we now construct an identity $h_{m_k}$, which is satisfied by the group $G$. Set $h_1=f_1$. If $[h_{i-1},f_i]=1$ in the free group, then there is word $w$, such that $h_{i-1}=w^{\alpha}$ and $f_i=w^{\beta}$ for some $\alpha$ and $\beta$ in $\mathbb{Z}$. Define $h_i=w^{\alpha\beta}$. If $[h_{i-1},f_i]\neq 1$ in the free group, then set $h_i=[h_{i-1},f_i]$. Consider the word $h_{m_k}$. For any $n$-tuple of elements $a_1,\dots,a_n\in G$ we have: $h_{m_k}(a_1,\dots,a_n)=1$. It follows that $h_{m_k}$ is an identity on $G$ - a contradiction.

\section{Identities in Thompson's group $F$}
By Lemma \ref{lem3}, it suffices to show that $F$ does not have an identity in two variables. Let $w(x,y)=w_k\dots w_2w_1$ be an arbitrary reduced non-trivial word in $x$ and $y$ of length $k$. That is $w_i \in \{x, x ^ {-1}, y, y ^ {-1} \} $, $1 \leq i \leq k $. We shall construct  elements $a$ and $b$ of Thompson's group $F$, for which $w (a, b) \neq 1 $.

Consider a partitioning $\beta_1 <\beta_2 <\dots <\beta _ {k+1}$ of $[0,1]$, where $\beta_i$ is a dyadic number. We represent such partitions of $[0,1]$ by cells. To differentiate such presentation from rectangular diagrams we draw the left side of the rectangle as a sinuous line:
$$
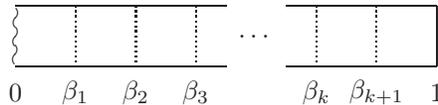
\xi
$$
\vspace {1mm}

Let $w(a,b) =u_k \dots u_2u_1 $, where $u_i \in \{a,a^{-1},b,b^{-1} \}$, $1\leq i\leq n$. Set $u_1(\beta_1) = \beta_2, \dots, u_k(\beta_k)=\beta_{k+1}$. Since the word $w(a,b)$ does not contain subwords of the form $a^{\epsilon}a^{-\epsilon}$, $b^\epsilon b^{-\epsilon}$, $\epsilon=\pm 1$, it follows that $a$ and $b$ are increasing functions on points on which they are defined.

\begin{expl}
Let $w (x, y) =y ^ {-1} x ^ {-1} yx $. Then the partially defined maps $a$ and $b$ have the following form:
$$
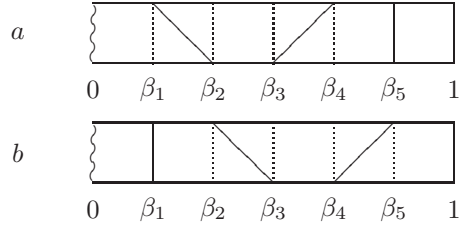
\xj
$$
\vspace {1mm}
$$
\xk
$$
\vspace {1mm}

The corresponding diagram for the word $w(a,b)=b^{-1}a^{-1}ba$ takes the form:
$$
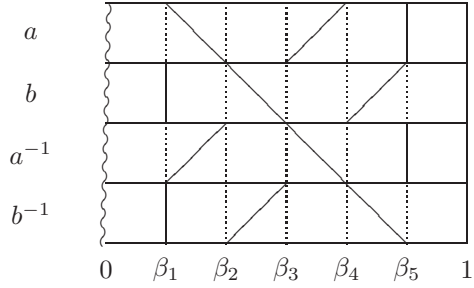
\xl
$$
\end{expl}
\vspace {1mm}

We use Lemma \ref{lem1} to define the functions $a$ and $b$ on the whole segment $[0,1]$.

Thus we have constructed two functions $a$ and $b$ so that  $w(a,b)(\beta_1)=\beta_{k+1}>\beta_1$. Therefore $w(a,b)$ is not the identity map, hence it does not represent the identity element in Thompson's group $F$. This finishes our proof of Lemma \ref{lem4}.

Theorem \ref{thm1} follows from Lemmas \ref{lem2} and \ref{lem4} and the fact that Thomson's group $F$ is torsion-free.

\section*{Appendix A. An alternative proof of Lemma \ref{lem1}}

Consider two adjacent points $x_i$ and $x_{i+1}$ of the partition $0=x_0<x_1<x_2<\dots <x_n=1$. Define the function $f$ on the segment $[x_i,x_{i+1}]$ as follows
$$
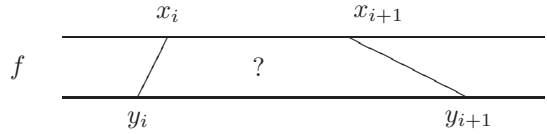
\xf
$$

Let $x_{i+1}-x_i=c_1$, $y_{i+1}-y_i=c_2$. Without loss of generality we may assume that $c_1<c_2$, where $c_1$ and $c_2$ are diadic numbers, that is
$c_1=\frac{t_1} {2^{j_1}}$, $c_2 = \frac{t_2}{2^{j_2}}$. Then $z_1=t_1 2^{j_2} $, $z_2=t_2 2^{j_1} $ are integers. It follows that
$d=\frac{z_1-1}{z_1}c_1=\frac{t_12^{j_2}-1}{t_12^{j_2}}\frac{t_1} {2^{j_1}}$ is a diadic number and $d<c_1$. Define the function $f$ on the segment
$[x_i,x_i+d]$ as a linear function, whose slope equals $1$:
$$
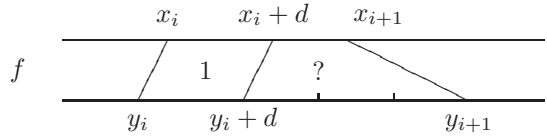
\xg
$$

The ratio of lengths of the remaining pieces equals
$$
\frac{c_2-d}{c_1-d}=\frac{z_2\frac{c_1}{z_1}-(z_1-1)\frac{c_1}{z_1}}{c_1-(z_1-1)\frac{c_1}{z_1}}=\frac {z_2-z_1+1}{1}.
$$
Write this number as a sum of powers of $2$, $z_2-z_1+1=2^{k_1} + \dots+2^{k_n} $, where $k_1,\dots,k_n$ are non-negative integers. Proportionally to this decomposition, we subdivide the segment $[y_i+d,y_{i+1}]$ into $m$ pieces. Furthermore, write $1$ as a sum of powers of $2$,
$1=2^{k'_1}+\dots+2^{k'_n}$, where $k'_1,\dots,k'_n $ are negative integers. We then subdivide proportionally the segment $[x_i+d,x_{i+1}]$ into $m$ pieces. We use the above defined partitions and define the function $f$ on the segment $[x_i+d,x_{i+1}]$ as a piecewise linear function.
It follows that the slopes are powers of $2$:
$$
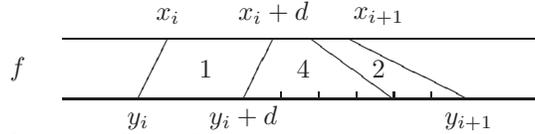
\xh
$$

Construct the function $f$ for all pairs $x_i,x_{i+1}$, $i=0,\dots,n$ as above. It follows that the function $f$ is defined on the whole segment $[0,1]$. This finishes our proof of Lemma \ref{lem1}

\end{document}